\documentclass{amsart}

\newtheorem{thm}{Theorem}[section]

\theoremstyle{definition}
\newtheorem{defn}[thm]{Definition}
\theoremstyle{remark}
\newtheorem{rem}[thm]{Remark}
\numberwithin{equation}{section}

\newcommand{\pom}{\partial\Omega}

\newcommand{\bfx}{\mbox{\boldmath{$x$}}}
\newcommand{\bfn}{\mbox{\boldmath{$n$}}}

\newcommand{\bfu}{\mbox{\boldmath{$u$}}}

\begin{document}

\title[Degenerate coupled flows in porous media]
{Global weak solutions to degenerate
coupled diffusion-convection-dispersion processes and
heat transport in porous media}%
\author{Michal Bene\v{s}}%
\address{Michal Bene\v{s}
\\
Department of Mathematics
\\
Faculty of Civil Engineering
\\
Czech Technical University in Prague
\\
Th\'{a}kurova 7, 166 29 Prague 6, Czech Republic}%
\email{michal.benes@cvut.cz}%
\author{Luk\'{a}\v{s} Krupi\v{c}ka}%
\address{Luk\'{a}\v{s} Krupi\v{c}ka
\\
Department of Mathematics
\\
Faculty of Civil Engineering
\\
Czech Technical University in Prague
\\
Th\'{a}kurova 7, 166 29 Prague 6, Czech Republic}%
\email{luk.krupicka@gmail.com}%

\thanks{}%

\begin{abstract}

In this contribution we prove the existence of weak solutions
to degenerate parabolic systems
arising from the coupled moisture movement,
transport of dissolved species and heat transfer
through partially saturated porous materials.
Physically motivated mixed Dirichlet-Neumann boundary conditions
and initial conditions are considered.
Existence of a global weak solution
of the problem is proved by means of semidiscretization in time
and by passing to the limit from discrete approximations.
Degeneration
occurs in the nonlinear transport coefficients
which are not assumed to be bounded below and above by positive
constants.
Degeneracies in all transport coefficients
are overcome by proving suitable a-priori $L^{\infty}$-estimates for the approximations
of primary unknowns of the system.

\end{abstract}

\maketitle

\section{Introduction}

Let $\Omega$ be a bounded domain in $\mathbb{R}^2$, $\Omega \in
C^{0,1}$  and let $\Gamma_D$ and $\Gamma_N$ be open disjoint subsets of
$\pom$ (not necessarily connected) such that $\Gamma_D\neq\emptyset$
and the $\partial\Omega \backslash (\Gamma_D\cup\Gamma_N)$ is a
finite set. Let $T\in(0,\infty)$ be fixed throughout the paper, $I=(0,T)$
and $Q_T = \Omega\times I$ denotes the space-time cylinder,
$\Gamma_{DT} = \Gamma_D\times I$
and
$\Gamma_{NT}=\Gamma_N\times I$.

We shall study the following initial boundary value problem in $Q_{T}$
\begin{align}
\partial_t b(u)
&
=
\nabla\cdot[a(\theta)\nabla u],
\label{strong:eq1a}
\\
\partial_t [b(u)w]
&
=
\nabla\cdot[b(u)D_w(u)\nabla w]
+
\nabla\cdot[ w a(\theta)\nabla u],
\label{strong:eq1b}
\\
\partial_t \left[ b(u) \theta + {\varrho} \theta  \right]
&=
\nabla \cdot [ \lambda(\theta,u) \nabla \theta ]
+
\nabla\cdot
[\theta a(\theta)\nabla u ],
\label{strong:eq1c}
\end{align}
with the mixed-type boundary conditions
\begin{align}
u &= 0,
\;
w = 0,
\;
\theta = 0
&&
{\rm on} \; \Gamma_{DT},
\label{strong:eq1f}
\\
{{\nabla u}}  \cdot \bfn
& = 0,
\;
\nabla w \cdot \bfn
=
0,
\;
\nabla \theta \cdot \bfn
=
0
&&
{\rm on} \;  \Gamma_{NT}
\label{strong:eq1i}
\end{align}
and the initial conditions
\begin{equation}
u(\cdot,0)  =  u_0,
\;
w(\cdot,0)  = w_0,
\;
\theta(\cdot,0)  = \theta_0
\qquad
{\rm in} \;  \Omega.
\label{strong:eq1l}
\end{equation}

The system \eqref{strong:eq1a}--\eqref{strong:eq1l} arises from the coupled moisture movement,
transport of dissolved species and heat transfer
through the porous system \cite{Bear,PinderGray}.
For simplicity,
the gravity terms and external sources are not included since they do not affect the analysis.
For specific applications we refer the reader to e.g. \cite{Ozbolt}.
Here $u : Q_{T} \rightarrow \mathbb{R}$, $w : Q_{T} \rightarrow \mathbb{R}$
and $\theta : Q_{T} \rightarrow \mathbb{R}$ are the unknown functions.
In particular, $u$ corresponds to the Kirchhoff transformation
of the matric potential \cite{AltLuckhaus1983}, $w$ represents concentration of dissolved species
and $\theta$ represents the temperature of the porous system.
Further, $a:\mathbb{R}\rightarrow \mathbb{R}$,
$D_w:\mathbb{R}\rightarrow \mathbb{R}$,
$b:\mathbb{R}\rightarrow \mathbb{R}$,
$\lambda:\mathbb{R}^2\rightarrow \mathbb{R}$,
$u_{0} : \Omega \rightarrow \mathbb{R}$,
$w_{0} : \Omega \rightarrow \mathbb{R}$,
and
$\theta_{0} : \Omega\rightarrow \mathbb{R}$
are given functions,
$\varrho$ is a real positive constant
and  $\bfn$ is the outward unit normal vector.
In this paper we study the existence
of the weak solution to the system  \eqref{strong:eq1a}--\eqref{strong:eq1l}.

\bigskip

Nowadays, description of heat, moisture or soluble/non-soluble
contaminant transport in concrete, soil or rock porous matrix is
frequently based on time dependent models.
Coupled transport processes (diffusion processes, heat conduction, moister flow,
contaminant transport or coupled flows through porous media) are typically associated with
systems of strongly nonlinear degenerate parabolic partial differential
equations of type (written in terms of operators
${A}$, ${\Psi}$, ${F}$)
\begin{equation}\label{eq:par1}
 \partial_t {\Psi}(\bfu) -\nabla\cdot
{A}(\bfu,\nabla \bfu)
=
{F}(\bfu),
\end{equation}
where $\bfu$ stands for the unknown vector of state variables.
There is no complete theory for
such general problems. However,
some particular results assuming special structure of
operators ${A}$ and ${\Psi}$ and growth
conditions on ${F}$ can be found in
the literature.
Most theoretical results on parabolic systems exclude the case of non-symmetrical parabolic parts \cite{AltLuckhaus1983,FiloKacur1995,Kacur1990a}.
Giaquinta and Modica in~\cite{GiaquintaModica1987} proved the local-in-time solvability of quasilinear
diagonal parabolic systems with nonlinear boundary conditions (without assuming any growth condition), see also \cite{Weidemaier1991}.
The existence of weak solutions to more general non-diagonal systems like \eqref{eq:par1} subject to mixed
boundary conditions has been proven in~\cite{AltLuckhaus1983}. The authors
proved an existence result assuming the operator ${\Psi}$
to be only (weak) monotone and subgradient. This
result has been extended in~\cite{FiloKacur1995}, where the authors presented
the local existence of the weak
solutions for the system with nonlinear Neumann boundary conditions
and under more general growth conditions on nonlinearities in $\bfu$.
These results, however, are not applicable if ${\Psi}$ does not take the
subgradient structure, which is typical of coupled transport models in porous media.
Thus, the analysis needs to exploit the specific structure of such problems.

The existence
of a local-in-time strong solution
for moisture and heat transfer in multi-layer porous structures
modelling by the doubly nonlinear parabolic system
is proven in \cite{BenesZeman}.
In \cite{Vala2002}, the author
proved the existence of the solution to the purely diffusive
hygro-thermal model allowing non-symmetrical operators $\Psi$, but requiring
non-realistic symmetry in the elliptic part.
In
\cite{degond,jungel2000}, the authors studied the existence, uniqueness and regularity
of coupled quasilinear equations modeling evolution of fluid species influenced by thermal, electrical
and diffusive forces.
In \cite{LiSun2010,LiSunWang2010,LiSun2012}, the authors studied a model of specific structure
of a heat and mass transfer arising from textile industry and
proved the global existence for one-dimensional problems in \cite{LiSun2010,LiSunWang2010}
and three-dimensional problems in \cite{LiSun2012}.

In the present paper we extend our previous existence result for coupled heat and
mass flows in porous media \cite{BenesKrupicka2016}
to more general problem (including the convection-dispersion equation)
modeling coupled moisture, solute and heat transport in porous media.
This leads to a fully nonlinear degenerate parabolic system
with natural (critical) growths and degeneracies in all transport coefficients.

The rest of this paper is organized as follows.  In Section \ref{sec:prelim},
we introduce basic notation and suitable function spaces
and specify our assumptions on data
and coefficient functions in the problem.
In Section~\ref{sec:main_result}, we formulate the problem in the
variational sense and state the main result, the global-in-time
existence of the weak solution.
The main result is proved by an approximation procedure in Section \ref{sec:proof_main}.
First we formulate the semi-discrete scheme and prove
the existence of its solution.
The crucial a-priori estimates and uniform boundness
of time interpolants are proved in part~\ref{sec:estimates}.
Finally, we conclude that
the solutions of semi-discrete scheme converge and the limit
is the solution of the original problem (Subsection~\ref{subsec:limit}).

\begin{rem}
The present analysis can be straightforwardly extended to a setting with
nonhomogeneous boundary conditions (see \cite{BenesKrupicka2016} for details).
Here we work with homogeneous boundary conditions, ignoring the gravity terms
and excluding external sources to simplify
the presentation and avoid unnecessary technicalities in the existence result.
\end{rem}

\section{Preliminaries}
\label{sec:prelim}
\subsection{Notations and some properties of Sobolev spaces}
\label{notations}
Vectors and vector functions are denoted by boldface letters.
Throughout the paper, we will always use positive constants $C$,
$c$, $c_1$, $c_2$, $\dots$, which are not specified and which may
differ from line to line.
Throughout this paper we suppose
$s,q,s'\in [1,\infty]$, $s'$ denotes the conjugate exponent to $s>1$,
${1}/{s} + {1}/{s'} = 1$.
$L^s(\Omega)$ denotes the usual
Lebesgue space equipped with the norm $\|\cdot\|_{L^s(\Omega)}$ and
$W^{k,s}(\Omega)$, $k\geq 0$ ($k$ need not to be an integer, see
\cite{KufFucJoh1977}), denotes the usual
Sobolev-Slobodecki space with the norm $\|\cdot\|_{W^{k,s}(\Omega)}$.
We define
$
W^{1,2}_{\Gamma_D}(\Omega)
:=
\left\{
v \in W^{1,2}(\Omega); \,   v \big|_{\Gamma_D} = 0
\right\}$.
By  $E^*$ we denote the space of all continuous, linear forms on Banach space $E$
and by $\langle \cdot,\cdot \rangle$ we denote the
duality between $E$ and $E^*$.
By $L^s(I;E)$ we denote the Bochner space (see \cite{AdamsFournier1992}).
Therefore, $L^s(I;E)^*=L^{s'}(I;E^*)$.

\subsection{Structure and data properties}
\label{Structure and data properties}

We start by introducing our assumptions on  functions in
\eqref{strong:eq1a}--\eqref{strong:eq1l}.
\begin{itemize}

\item[(i)]
$b$ is a positive continuous strictly monotone function such that
\begin{align*}
& 0 < b(\xi) \leq b_{2} < +\infty
&&
\forall \xi \in \mathbb{R} \quad (b_{2} = {\rm const}),
\\
&  \left(b(\xi_1)-b(\xi_2)\right)(\xi_1 - \xi_2) > 0
&&
\forall \xi_1,\xi_2 \in \mathbb{R}, \;  \xi_1 \neq \xi_2.
\end{align*}

\item[(ii)]
$a$, $D_w$ $\in C(\mathbb{R})$ and $\lambda$ $\in C(\mathbb{R}^2)$ such that
\begin{align*}
&
0 < a(\xi) , \; 0 <  D_w(\xi)
&&
\forall \xi \in \mathbb{R},
\\
&
0 < \lambda(\xi,\zeta)
&&
\forall \xi,\zeta \in \mathbb{R}.
\end{align*}

\item[(iii)] (Initial data)
Assume
\begin{displaymath}
u_0, \;  w_0, \; \theta_0 \in L^{\infty}(\Omega),
\end{displaymath}
such that
\begin{equation}
-\infty < u_{1} < u_0 < u_{2} < 0   \qquad \textmd{ a.e. in } \Omega \quad
(u_{1},u_{2}={\rm const}).
\end{equation}

\end{itemize}

\subsection{Auxiliary results}
\begin{rem}[\cite{AltLuckhaus1983}, Section 1.1]
Let us note that {\rm (i)} implies that
there is a (strictly) convex $C^1$-function
$\Phi:\mathbb{R}\rightarrow \mathbb{R}$, $\Phi(0)=0$,
$\Phi'(0)=0$, such that
$b(z) - b(0) = \Phi'(z) \; \forall  z \in \mathbb{R}$.
Introduce the Legendre transform
$$
B(z)
:=
\int_{0}^{1}(b(z)-b(sz)) z \, {\rm d}s
=
\int_{0}^{z}(b(z)-b(s))\, {\rm d}s .
$$
\end{rem}
Let us present some properties of $B$ \cite{AltLuckhaus1983}:
\begin{align*}
&
B(z)
:=
\int_{0}^{1}(b(z)-b(sz)) z \, {\rm d}s \geq 0
&&
\forall z \in \mathbb{R},
\\
&
B(s) - B(r) \geq (b(s)-b(r))r
&&
\forall r,s \in \mathbb{R},
\\
&
b(z)z - \Phi(z) + \Phi(0) = B(z) \leq b(z)z
&&
\forall z \in \mathbb{R}.
\end{align*}

\section{The main result}\label{sec:main_result}
The aim of this paper is to prove the existence of a weak solution to the problem
\eqref{strong:eq1a}--\eqref{strong:eq1l}.
First we formulate our problem in a variational sense.
\begin{defn}
\label{def_weak_solution}
A weak solution of \eqref{strong:eq1a}--\eqref{strong:eq1l}
is a triplet $[u,w,\theta]$ such that
\begin{equation*}
u \in L^2(I;W_{\Gamma_D}^{1,2}(\Omega)) ,
\,
w \in  L^2(I;W_{\Gamma_D}^{1,2}(\Omega)) \cap L^{\infty}({Q_T}),
\,
\theta \in  L^2(I;W_{\Gamma_D}^{1,2}(\Omega)) \cap L^{\infty}({Q_T}),
\end{equation*}
which satisfies
\begin{equation}
\label{weak_form_01}
-\int_{Q_T}  b(u) \partial_t\phi
{\,{\rm d}x} {\rm d}t
+
\int_{Q_T}
{a(\theta){\nabla u}} \cdot\nabla\phi
{\,{\rm d}x} {\rm d}t
=
\int_{\Omega}   b(u_{0}) \phi(\bfx,0) {\,{\rm d}x}
\end{equation}
for any  $\phi \in L^2(I;W^{1,2}_{\Gamma_D}(\Omega))\cap W^{1,1}(I;L^{\infty}(\Omega))$
with $\phi(\cdot,T)=0$;

\begin{multline}
\label{weak_form_02}
-\int_{Q_T}  b(u)w \partial_t\eta
{\,{\rm d}x} {\rm d}t
+
\int_{Q_T}
b(u)D_w(u)\nabla w \cdot\nabla \eta
{\,{\rm d}x} {\rm d}t
\\
+
\int_{Q_T}
w {a(\theta){\nabla u}} \cdot\nabla\eta
{\,{\rm d}x} {\rm d}t
=
\int_{\Omega} b(u_0)w_0 \eta(\bfx,0) {\,{\rm d}x}
\end{multline}
for any  $\eta \in L^2(I;W^{1,2}_{\Gamma_D}(\Omega))\cap W^{1,1}(I;L^{\infty}(\Omega))$
with $\eta(\cdot,T)=0$;

\begin{multline}\label{weak_form_03}
-\int_{Q_T}
[  b(u)\theta + {\varrho} \theta ] \partial_t\psi
{\,{\rm d}x} {\rm d}t
+
\int_{Q_T}
\lambda(\theta,u) \nabla\theta \cdot \nabla\psi
{\,{\rm d}x} {\rm d}t
\\
+
\int_{Q_T}
\theta {a(\theta){\nabla u}}
\cdot \nabla\psi
{\,{\rm d}x} {\rm d}t
=
\int_{\Omega}
[  b(u_{0}) \theta_{0} + {\varrho}  \theta_{0} ] \psi(\bfx,0) {\,{\rm d}x}
\end{multline}
for any  $\psi \in L^2(I;W^{1,2}_{\Gamma_D}(\Omega))\cap W^{1,1}(I;L^{\infty}(\Omega))$
with $\psi(\cdot,T)=0$.
\end{defn}

The main result of this paper reads as follows.
\begin{thm}[Main result]\label{main_result}
Let the assumptions {\rm (i)--(iii)} be satisfied.
Then there exists at
least one weak solution of the
system  \eqref{strong:eq1a}--\eqref{strong:eq1l}.
\end{thm}

To prove the main result of the paper we use the method
of semidiscretization in time by constructing temporal approximations
and limiting procedure.
The proof can be divided into three steps.
In the first step we approximate our problem by means of a semi-implicit time discretization
scheme (which preserve the pseudo-monotone structure of the discrete problem)
and prove the existence and $W^{1,s}(\Omega)$-regularity (with some $s>2$) of temporal approximations.
In the second step we construct piecewise constant time interpolants and derive suitable a-priori estimates.
The key point is to establish $L^{\infty}$-estimates to overcome degeneracies in transport coefficients.
Finally, in the third step we pass to the limit from discrete approximations.

\section{Proof of the main result}\label{sec:proof_main}
\subsection{Approximations}\label{sec:approximations}

Let us fix $p \in \mathbb{N}$ and set $\tau:= T/p$ (a time step).
Further, let us consider
$u^0_{p} := u_{0}$,
$w^0_{p} := w_{0}$,
$\theta^{0}_{p} := \theta_{0}$ a.e. on $\Omega$.
We approximate our evolution problem by a semi-implicit time discretization scheme.
Then we define, in each time step $n=1,\dots,p$, a triplet $[u^n_{p},w^n_{p},\theta^n_{p}]$
as a solution of the
following recurrence steady problem.


\bigskip

For a given triplet
$[u^{n-1}_{p},w^{n-1}_{p},\theta^{n-1}_{p}]$,
$n=1,\dots,p$,
$u^{n-1}_{p} \in  L^{\infty}(\Omega)$,
$w^{n-1}_{p} \in L^{\infty}(\Omega)$,
$\theta^{n-1}_{p} \in  L^{\infty}(\Omega)$,
find  $[u^n_{p},w^n_{p},\theta^n_{p}]$,
such that
$u^n_{p} \in  W_{\Gamma_D}^{1,s}(\Omega)$,
$w^n_{p} \in  W_{\Gamma_D}^{1,s}(\Omega)$,
$\theta^n_{p} \in  W_{\Gamma_D}^{1,s}(\Omega)$ with some $s>2$
and
\begin{equation}
\label{approximate_problem_01}
\int_{\Omega}
\frac{b(u_{p}^n) - b(u_{p}^{n-1})}{\tau} \phi
{\,{\rm d}x}
+
\int_{\Omega}
a(\theta_{p}^{n-1}) \nabla u_{p}^n
\cdot\nabla\phi
{\,{\rm d}x}
=
0
\end{equation}
for any $\phi \in {W_{\Gamma_D}^{1,2}(\Omega)}$;

\begin{multline}\label{approximate_problem_02}
\int_{\Omega}
\frac{ b(u_{p}^n)w^n_{p} - b(u_{p}^{n-1})w_{p}^{n-1} }{\tau} \eta
{\,{\rm d}x}
\\
+
\int_{\Omega}
b(u_{p}^{n-1})D_w(u_{p}^{n-1})
\nabla w_{p}^n \cdot \nabla\eta
{\,{\rm d}x}
+
\int_{\Omega}
w_{p}^n a(\theta_{p}^{n-1}) \nabla u_{p}^n
\cdot \nabla \eta
{\,{\rm d}x}
=
0
\end{multline}
for any  $\eta \in {W_{\Gamma_D}^{1,2}(\Omega)}$;

\begin{multline}\label{approximate_problem_03}
\int_{\Omega}
\frac{ b(u_{p}^n)\theta^n_{p} - b(u_{p}^{n-1})\theta_{p}^{n-1} }{\tau} \psi
{\,{\rm d}x}
+
\int_{\Omega}
{\varrho}
\frac{ \theta_{p}^n - \theta_{p}^{n-1} }{\tau}  \psi
{\,{\rm d}x}
\\
+
\int_{\Omega}
\lambda(\theta_{p}^{n-1},u_{p}^{n-1}) \nabla \theta_{p}^n \cdot \nabla\psi
{{\rm d}\Omega}
+
\int_{\Omega}
\theta_{p}^n a(\theta_{p}^{n-1}) \nabla u_{p}^n
\cdot \nabla\psi
{{\rm d}\Omega}
=
0
\end{multline}
for any $\psi \in {W_{\Gamma_D}^{1,2}(\Omega)}$.

\begin{thm}[Existence of the solution to
\eqref{approximate_problem_01}--\eqref{approximate_problem_03}]
\label{thm:aprox}
Let
$u^{n-1}_{p} \in  L^{\infty}(\Omega)$,
$w^{n-1}_{p} \in L^{\infty}(\Omega)$,
$\theta^{n-1}_{p} \in  L^{\infty}(\Omega)$
be given and the assumptions {\rm (i)--(iii)} be satisfied.
Then there exists $[u^n_{p},w^n_{p},\theta^n_{p}]$,
such that
$u^n_{p} \in  W_{\Gamma_D}^{1,s}(\Omega)$,
$w^n_{p} \in  W_{\Gamma_D}^{1,s}(\Omega)$
and
$\theta^n_{p} \in  W_{\Gamma_D}^{1,s}(\Omega)$ with some $s>2$
satisfying
\eqref{approximate_problem_01}--\eqref{approximate_problem_03}.
\end{thm}
\begin{proof}
The proof rests on the
$W^{1,p}$-regularity of elliptic problems presented in \cite{gallouet,groger1989}
and the embedding
$W_{\Gamma_D}^{1,s}(\Omega) \subset L^{\infty}(\Omega)$ if $s>2$
(recall that  $\Omega$ is a bounded domain in $\mathbb{R}^2$).

The existence of $u^n_{p} \in  W_{\Gamma_D}^{1,s}(\Omega)$ with some $s>2$
and $\theta^n_{p} \in  W_{\Gamma_D}^{1,2}(\Omega)$, solutions to problems
\eqref{approximate_problem_01} and \eqref{approximate_problem_03}, respectively,
is proven in \cite{BenesKrupicka2016}.
The existence of  $w^n_{p} \in  W_{\Gamma_D}^{1,2}(\Omega)$,
the solution to \eqref{approximate_problem_02},
can be handled in the same way.

Now, with $w^n_{p} \in  W_{\Gamma_D}^{1,2}(\Omega)$ in hand,
rewrite the equation \eqref{approximate_problem_02} in the form
(transferring the lower-order terms to the right hand side)
\begin{multline*}
\int_{\Omega}
b(u_{p}^{n-1})D_w(u_{p}^{n-1})
\nabla w_{p}^n \cdot \nabla\eta
{\,{\rm d}x}
\\
=
-
\int_{\Omega}
\frac{ b(u_{p}^n)w^n_{p} - b(u_{p}^{n-1})w_{p}^{n-1} }{\tau} \eta
{\,{\rm d}x}
-
\int_{\Omega}
w_{p}^n a(\theta_{p}^{n-1}) \nabla u_{p}^n
\cdot \nabla \eta
{\,{\rm d}x}.
\end{multline*}
Since $u^{n-1}_{p} \in L^{\infty}(\Omega)$, $u^n_{p} \in  W_{\Gamma_D}^{1,s}(\Omega)$ with some $s>2$,
$w^{n-1}_{p} \in L^{\infty}(\Omega)$,
$\theta^{n-1}_{p} \in  L^{\infty}(\Omega)$,
both integrals on the right hand side
make sense for any $\eta \in W_{\Gamma_D}^{1,r'}(\Omega)$, $r'=r/(r-1)$ with some $r>2$.
Now we are able to apply \cite[Theorem~4]{gallouet} to obtain
$w^n_{p} \in  W_{\Gamma_D}^{1,s}(\Omega)$ with some $s>2$.
Analysis similar to the above implies that
$\theta^n_{p} \in  W_{\Gamma_D}^{1,s}(\Omega)$ with some $s>2$.
\end{proof}

\subsection{A-priori estimates}
\label{sec:estimates}
In this part we prove some uniform estimates (with respect to $p$) for
the time interpolants of the solution.
In the following estimates, many different constants
will appear. For simplicity of notation,
$C$ represents generic constants which may change their numerical value from one formula to another
but do not depend on $p$ and
the functions under consideration.

\subsubsection{Construction of temporal interpolants}
With the sequences $u^n_{p},w^n_{p},\theta^n_{p}$ constructed
in Section~\ref{sec:approximations}, we define
the piecewise constant interpolants
$\bar{\phi}_{p}(t) = \phi_{p}^{n}$ for $t \in ((n - 1)\tau, n\tau]$
and, in addition, we extend $\bar{\phi}_{p}$ for $t\leq 0$ by
$
\bar{\phi}_{p}(t) = \phi_{0}$ for $t \in (-\tau, 0]$.
For a function $\varphi$ we often use the simplified notation
$\varphi := \varphi(t)$, $\varphi_\tau(t) := \varphi(t-\tau)$,
$\partial_t^{-\tau}\varphi(t) := \frac{\varphi(t) - \varphi(t-\tau)}{\tau}$,
$\partial_t^{\tau}\varphi(t) := \frac{\varphi(t+\tau) - \varphi(t)}{\tau}$.
Then,
following \eqref{approximate_problem_01}--\eqref{approximate_problem_03},
the piecewise constant time interpolants
$\bar{u}_{p} \in L^{\infty}(I;W_{\Gamma_D}^{1,s}(\Omega))$,
$\bar{w}_{p} \in L^{\infty}(I;W_{\Gamma_D}^{1,s}(\Omega))$
and
$\bar{\theta}_{p} \in L^{\infty}(I;W_{\Gamma_D}^{1,s}(\Omega))$ (with some $s>2$)
satisfy the equations
\begin{equation}
\label{eq:101}
\int_{\Omega}
\partial_t^{-\tau} {b}(\bar{u}_{p}(t)) \phi
{\,{\rm d}x}
+
\int_{\Omega}
a(\bar{\theta}_{p}(t-\tau))\nabla \bar{u}_{p}(t)
\cdot\nabla \phi
{\,{\rm d}x}
=
0
\end{equation}
for any $\phi \in {W_{\Gamma_D}^{1,2}(\Omega)}$,
\begin{multline}\label{eq:102}
\int_{\Omega}
\partial_t^{-\tau}  [ {b}(\bar{u}_{p}(t))\bar{w}_{p}(t) ] \eta
{\,{\rm d}x}
+
\int_{\Omega}
{b}(\bar{u}_{p}(t-\tau))D_w(\bar{u}_{p}(t-\tau)) \nabla \bar{w}_{p}(t) \cdot \nabla\eta
{\,{\rm d}x}
\\
+
\int_{\Omega}
\bar{w}_{p}(t) a(\bar{\theta}_{p}(t-\tau))\nabla \bar{u}_{p}(t)
\cdot \nabla\eta
{\,{\rm d}x}
=
0
\end{multline}
for any $\eta \in {W_{\Gamma_D}^{1,2}(\Omega)}$
and

\begin{multline}\label{eq:103}
\int_{\Omega}
\partial_t^{-\tau} \left[ {b}(\bar{u}_{p}(t))\bar{\theta}_{p}(t)
+
{\varrho}\bar{\theta}_{p}(t)
\right] \psi
{\,{\rm d}x}
\\
+
\int_{\Omega}
\lambda(\bar{\theta}_{p}(t-\tau),\bar{u}_{p}(t-\tau)) \nabla \bar{\theta}_{p}(t) \cdot \nabla\psi
{\,{\rm d}x}
\\
+
\int_{\Omega}
\bar{\theta}_{p}(t)
a(\bar{\theta}_{p}(t-\tau))\nabla \bar{u}_{p}(t)
\cdot \nabla\psi
{\,{\rm d}x}
=
0
\end{multline}
for any $\psi \in {W_{\Gamma_D}^{1,2}(\Omega)}$.


\subsubsection{$L^{\infty}$-bound for $\bar{u}_{p}$, $\bar{w}_{p}$ and $\bar{\theta}_{p}$}

First we prove the $L^{\infty}$-estimate for $\bar{u}_{p}$.
Let us set
\begin{equation}
\phi := [ b(\bar{u}_{p}) - b(u_{1}) ]_{-}
=
\left\{
\begin{array}{ll}
b(\bar{u}_{p}) - b(u_{1}) ,         &  \bar{u}_{p} < u_{1} ,
\\
0 ,                   & \bar{u}_{p} \geq u_{1},
\end{array} \right.
\end{equation}
as a test function in  \eqref{eq:101}. Note that $\phi$ vanishes on $\Gamma_D$.
It is a simple matter to deduce
\begin{multline}
\label{}
\frac{1}{2}\int_{\Omega} [ b(\bar{u}_{p}(t))-b(u_{1}) ]_-^2{\rm d}{x}
\\
+
\int_{Q_t}
a(\bar{\theta}_{p}(s-\tau)) b'(\bar{u}_{p}(s)) |\nabla \bar{u}_{p}(s)|^2
\chi_{\left\{ \bar{u}_{p} < u_{1}  \right\}}
{\rm d}x{\rm d}s
=
0
\end{multline}
for almost every $t\in I$.
Hence we conclude that the set $\left\{x \in \Omega;\; \bar{u}_{p}(x,t) < u_{1}   \right\}$
has a measure zero for almost every $t\in I$.

Now setting
\begin{equation}
\phi = [ b(\bar{u}_{p})-b(u_{2}) ]_{+}
=
\left\{
\begin{array}{ll}
b(\bar{u}_{p}) - b(u_{2}) ,         &  \bar{u}_{p} > u_{2} ,
\\
0 ,                   & \bar{u}_{p} \leq u_{2},
\end{array} \right.
\end{equation}
we obtain, using similar arguments,
$$
\frac{1}{2}\int_{\Omega} [ b(\bar{u}_{p}) - b(u_{2}) ]_+^2{\rm d}{x} = 0
\quad \textmd{ for almost every }t\in I.
$$
Hence the set $\left\{x \in \Omega;\; \bar{u}_{p}(x,t) > u_{2}   \right\}$
has a measure zero for almost every $t\in I$.
Finally, combining the previous arguments, we deduce
\begin{equation}\label{est:uniform_bound_u_02}
\|\bar{u}_{p}\|_{L^{\infty}({Q_T})} \leq {C},
\end{equation}
where ${C}$ does not depend on $p$.

Now we prove a similar estimate for $\bar{w}_{p}$.
Let $\ell$ be an odd integer.
Using $\phi = [\ell/(\ell+1)] (\bar{w}_{p})^{\ell+1}$
as a test function in \eqref{eq:101}
and
$\eta = (\bar{w}_{p})^{\ell}$ in \eqref{eq:102}
and combining both equations we obtain
\begin{align}\label{eq:501}
&
\frac{1}{\tau}
\frac{1}{\ell+1}
\int_{\Omega}
 {b}(\bar{u}_{p}(s))[\bar{w}_{p}(s)]^{\ell+1}
{\,{\rm d}x}
-
\frac{1}{\tau}
\frac{1}{\ell+1}
\int_{\Omega}
 {b}(\bar{u}_{p}(s-\tau))[\bar{w}_{p}(s-\tau)]^{\ell+1}
{\,{\rm d}x}
\\
&
+
\frac{1}{\tau}
\frac{1}{\ell+1}
\int_{\Omega}
 {b}(\bar{u}_{p}(s-\tau))[\bar{w}_{p}(s-\tau)]^{\ell+1}
{\,{\rm d}x}
+
\frac{1}{\tau}
\frac{\ell}{\ell+1}
\int_{\Omega}
 {b}(\bar{u}_{p}(s-\tau))[\bar{w}_{p}(s)]^{\ell+1}
{\,{\rm d}x}
\nonumber
\\
&
-
\frac{1}{\tau}
\int_{\Omega}
 {b}(\bar{u}_{p}(s-\tau))\bar{w}_{p}(s-\tau)[\bar{w}_{p}(s)]^{\ell}
{\,{\rm d}x}
\nonumber
\\
&
+
\int_{\Omega}
\ell [\bar{w}_{p}(s)]^{\ell-1}
{b}(\bar{u}_{p}(s-\tau))D_w(\bar{u}_{p}(s-\tau)) \nabla \bar{w}_{p}(s) \cdot \nabla\bar{w}_{p}(s)
{\,{\rm d}x}
\nonumber
\\
&
=
0.
\nonumber
\end{align}
Applying the Young's inequality we can write for the term in the third line
\begin{multline}\label{eq:502}
\frac{1}{\tau}
\int_{\Omega}
 {b}(\bar{u}_{p}(s-\tau))\bar{w}_{p}(s-\tau)[\bar{w}_{p}(s)]^{\ell}
{\,{\rm d}x}
\leq
\frac{1}{\tau}
\frac{1}{\ell+1}
\int_{\Omega}
 {b}(\bar{u}_{p}(s-\tau))[\bar{w}_{p}(s-\tau)]^{\ell+1}
{\,{\rm d}x}
\\
+
\frac{1}{\tau}
\frac{\ell}{\ell+1}
\int_{\Omega}
 {b}(\bar{u}_{p}(s-\tau))[\bar{w}_{p}(s)]^{\ell+1}
{\,{\rm d}x}.
\end{multline}
Combining \eqref{eq:501} and \eqref{eq:502} we deduce
\begin{align}\label{eq:503}
&
\frac{1}{\tau}
\frac{1}{\ell+1}
\int_{\Omega}
 {b}(\bar{u}_{p}(s))[\bar{w}_{p}(s)]^{\ell+1}
{\,{\rm d}x}
-
\frac{1}{\tau}
\frac{1}{\ell+1}
\int_{\Omega}
 {b}(\bar{u}_{p}(s-\tau))[\bar{w}_{p}(s-\tau)]^{\ell+1}
{\,{\rm d}x}
\\
&
+
\int_{\Omega}
\ell [\bar{w}_{p}(s)]^{\ell-1}
{b}(\bar{u}_{p}(s-\tau))D_w(\bar{u}_{p}(s-\tau)) \nabla \bar{w}_{p}(s) \cdot \nabla\bar{w}_{p}(s)
{\,{\rm d}x}
\nonumber
\\
&
\leq
0.
\nonumber
\end{align}
Now, integrating \eqref{eq:503} over $s$ from $0$ to $t$ we get
\begin{align}\label{app:eq_bound_w_00}
&
\int_{\Omega}
(\bar{w}_{p}(t))^{\ell+1}
b(\bar{u}_{p}(t))
{\rm d}x
\\
&
+
\int_{\Omega_t}
(\ell+1)\ell [\bar{w}_{p}(s)]^{\ell-1}
{b}(\bar{u}_{p}(s-\tau))D_w(\bar{u}_{p}(s-\tau))
|\nabla \bar{w}_{p}(s)|^2
{\rm d}x{\rm d}s
\nonumber
\\
&
\leq
\int_{\Omega}
(w_0)^{\ell+1}
b\left(u_0\right)
{\rm d}{x}.
\nonumber
\end{align}
Note that the second integral in
\eqref{app:eq_bound_w_00} is nonnegative ($\ell$ is supposed to be the odd integer).
Moreover, from \eqref{app:eq_bound_w_00} and \eqref{est:uniform_bound_u_02} it follows that
\begin{equation}\label{est:unform_bound_w_01}
\|\bar{w}_{p}\|_{L^{\infty}(0,T;L^{\ell+1}(\Omega))} \leq C,
\end{equation}
where the constant $C$ is independent of $\ell$ and $p$.
Now, let $\ell \rightarrow +\infty$ in \eqref{est:unform_bound_w_01}, we get
\begin{equation}\label{est:unform_bound_w_05}
\|\bar{w}_{p}\|_{L^{\infty}({Q_T})} \leq C.
\end{equation}

In the same manner we arrive at
the estimate for $\bar{\theta}_{p}$, i.e.
\begin{equation}\label{est:unform_bound_theta_05}
\|\bar{\theta}_{p}\|_{L^{\infty}({Q_T})} \leq C.
\end{equation}

\subsubsection{Energy estimates for $\bar{u}_{p}$, $\bar{w}_{p}$ and $\bar{\theta}_{p}$}
We test \eqref{eq:101} with $\phi = \bar{u}_{p}(t)$ and integrate \eqref{eq:101} over $t$
from $0$ to $s$. For the parabolic term we can write
\begin{equation}\label{est:301}
\int_0^s
\int_{\Omega}
\partial_t^{-\tau} {b}(\bar{u}_{p}(t))  \bar{u}_{p}(t)
{\,{\rm d}x}{\rm d}t
\geq
\frac{1}{\tau}
\int_{s-\tau}^s
\int_{\Omega}
{B}(\bar{u}_{p}(t)) - {B}( u_0 )
{\,{\rm d}x}{\rm d}t.
\end{equation}
Further,  using \eqref{est:uniform_bound_u_02} and \eqref{est:301},
applying the usual estimates for the elliptic
part (see also \cite{AltLuckhaus1983}), we obtain the a-priori estimate
\begin{equation}\label{energy_estimate_u}
\sup_{0\leq t \leq T} \int_{\Omega}{B}(\bar{u}_{p}(t)) {\rm d}x
+
\int_0^T\int_{\Omega}  |\nabla \bar{u}_{p}(t)|^2 {\rm d}x{\rm d}t
\leq {C}.
\end{equation}
Now it follows that there exists a function
$u \in L^2(I;W^{1,2}_{\Gamma_D}(\Omega))$
such that, along a selected
subsequence (letting $p \rightarrow \infty$), we have
$\bar{u}_{p}(t)  \rightharpoonup  u$ weakly in $L^2(I;W^{1,2}_{\Gamma_D}(\Omega))$.

Now we prove similar result for $\bar{w}_{p}(t) $.
Using $\eta(t) = 2 \bar{w}_{p}(t) $
as a test function in \eqref{eq:102} we get
\begin{align}\label{est:_w_01}
&
\int_{\Omega}
\partial_t^{-\tau}{b}(\bar{u}_{p}(t))  2 \bar{w}_{p}(t)^2
{\,{\rm d}x}
+
\int_{\Omega}
\partial_t^{-\tau} \bar{w}_{p}(t) 2\bar{w}_{p}(t)
b(\bar{u}_{p}(t-\tau))
{\,{\rm d}x}
\nonumber
\\
&
+
2\int_{\Omega}
{b}(\bar{u}_{p}(t-\tau))D_w(\bar{u}_{p}(t-\tau)) \nabla \bar{w}_{p}(t)
\cdot
\nabla\bar{w}_{p}(t)
{\,{\rm d}x}
\nonumber
\\
&
+
\int_{\Omega}
a(\bar{\theta}_{p}(t-\tau)) \nabla \bar{u}_{p}(t)
\cdot
2\bar{w}_{p}(t) \nabla \bar{w}_{p}(t)
{\,{\rm d}x}
=0.
\end{align}
One is allowed to use $\phi(t) = \bar{w}_{p}(t)^2$
as a test function in \eqref{eq:101} to obtain
\begin{equation}\label{est:_w_02}
\int_{\Omega}
[\partial_t^{-\tau} {b}(\bar{u}_{p}(t))]  \bar{w}_{p}(t)^2{\,{\rm d}x}
+
\int_{\Omega}
a(\bar{\theta}_{p}(t-\tau)) \nabla \bar{u}_{p}(t)
\cdot\nabla \bar{w}_{p}(t)^2
{\,{\rm d}x}
=
0.
\end{equation}
Combining \eqref{est:_w_01} and \eqref{est:_w_02} we deduce
\begin{align}\label{est:w_100}
&
\int_{\Omega}
\partial_t^{-\tau}
\left[
\bar{w}_{p}(t) ^2
b(\bar{u}_{p}(t))
\right]
{\,{\rm d}x}
%
+
\int_{\Omega}
\frac{1}{\tau}
\left[
\bar{w}_{p}(t)
-
\bar{w}_{p}(t-\tau)
\right]^2
b(\bar{u}_{p}(t-\tau))
{\,{\rm d}x}
\\
&
+
2\int_{\Omega}
{b}(\bar{u}_{p}(t-\tau))D_w(\bar{u}_{p}(t-\tau)) \nabla \bar{w}_{p}(t) \cdot \nabla \bar{w}_{p}(t)
{\,{\rm d}x}
=
0.
\nonumber
\end{align}
In view of \eqref{est:uniform_bound_u_02} we have
\begin{align}
b(\bar{u}_{p}(t)), \;
{b}(\bar{u}_{p}(t-\tau)), \;
D_w(\bar{u}_{p}(t-\tau)) > C \qquad \textmd{ in } \Omega \times (-\tau,T).
\end{align}
Recall that $C$ does not depend on $p$.
Now, integrating \eqref{est:w_100} with respect to time $t$ we obtain
\begin{equation*}
\sup_{0\leq t \leq T}
\int_{\Omega} |\bar{w}_{p}(t)|^2  {\rm d}\Omega
+
\int_0^T \|\bar{w}_{p}(t)\|^2_{W^{1,2}_{\Gamma_D}(\Omega)} {\rm d}\Omega
\leq {C}.
\end{equation*}
From this we can write
\begin{equation}\label{est20_w_a}
\| \bar{w}_p \|_{L^2(I;W^{1,2}_{\Gamma_D}(\Omega))}  \leq  {C}.
\end{equation}

Similarly, we use $\psi(t) = 2\bar{\theta}_{p}(t)$
as a test function in \eqref{eq:102} to obtain
\begin{align}\label{approximate_estimate_02}
&
\int_{\Omega}
\partial_t^{-\tau} b(\bar{u}_{p}(t))2 \bar{\theta}_{p}(t)^2
{\,{\rm d}x}
+
\int_{\Omega}
\partial_t^{-\tau} \bar{\theta}_{p}(t) 2\bar{\theta}_{p}(t)
b(\bar{u}_{p}(t-\tau))
{\,{\rm d}x}
\\
&
+
2
\int_{\Omega}
\lambda(\bar{\theta}_{p}(t-\tau),\bar{u}_{p}(t-\tau))
\nabla \bar{\theta}_{p}(t)
\cdot\nabla \bar{\theta}_{p}(t)
{\,{\rm d}x}
\nonumber
\\
&
+
\int_{\Omega}
a(\bar{\theta}_{p}(t-\tau)) \nabla \bar{u}_{p}(t)
\cdot 2 \bar{\theta}_{p}(t) \nabla\bar{\theta}_{p}(t)
{\,{\rm d}x}
=
0.
\nonumber
\end{align}
Using $\phi(t) = \bar{\theta}_{p}(t)^2$
as a test function in \eqref{eq:101} we get
\begin{equation}\label{approximate_estimate_03}
\int_{\Omega}
\partial_t^{-\tau} {b}(\bar{u}_{p}(t))  \bar{\theta}_{p}(t)^2
{\,{\rm d}x}
+
\int_{\Omega}
a(\bar{\theta}_{p}(t-\tau)) \nabla \bar{u}_{p}(t)
\cdot\nabla \bar{\theta}_{p}(t)^2
{\,{\rm d}x}
=0.
\end{equation}
Combining \eqref{approximate_estimate_02} and \eqref{approximate_estimate_03} we deduce
\begin{multline}\label{approximate_estimate_05}
\int_{\Omega}
\partial_t^{-\tau}
\left[
\left( \bar{\theta}_{p}(t) \right)^2
b(\bar{u}_{p}(t))
\right]
{\,{\rm d}x}
+
\int_{\Omega}
\frac{1}{\tau}\left[
\bar{\theta}_{p}(t)
-
\bar{\theta}_{p}(t-\tau)
\right]^2
b(\bar{u}_{p}(t-\tau))
{\,{\rm d}x}
\\
+
2
\int_{\Omega}
\lambda(\bar{\theta}_{p}(t-\tau),\bar{u}_{p}(t-\tau))
\nabla \bar{\theta}_{p}(t)
\cdot\nabla \bar{\theta}_{p}(t)
{\,{\rm d}x}
=
0.
\end{multline}
Integrating \eqref{approximate_estimate_05} with respect to time $t$  we obtain the a-priori estimate
(using \eqref{est:uniform_bound_u_02} and \eqref{est:unform_bound_theta_05})
\begin{equation}\label{apriori_est_theta_01}
\sup_{0\leq t \leq T} \int_{\Omega} |\bar{\theta}_{p}(t)|^2 {\rm d}x
+
\int_0^T \|\bar{\theta}_{p}(t)\|^2_{W^{1,2}_{\Gamma_D}(\Omega)} {\rm d}t
\leq {C}.
\end{equation}
From this we have
\begin{equation}\label{est20_theta_a}
\| \bar{\theta}_p \|_{L^2(I;W^{1,2}_{\Gamma_D}(\Omega))} \leq {C}.
\end{equation}

\subsubsection{Further estimates}

In order to show that $\bar{u}_{p}$ converges to $u$
almost everywhere on $Q_T$ we follow \cite{AltLuckhaus1983}.
Let $k \in \mathbb{N}$ and use
$$
\phi(t) = \partial^{k\tau}_t \bar{u}_{p}(s)
$$
for $j\tau \leq t \leq (j+k)\tau$ with $(j-1)\tau \leq s \leq j\tau$ and $1\leq j\leq\frac{T}{\tau}-k$,  as a test function in  \eqref{eq:101}.
For the parabolic term, we can write
\begin{align*}
&\int_{j \tau}^{(j+k)\tau}
\int_{\Omega}
\partial_t^{-\tau} {b}(\bar{u}_{p}(t)) \, \partial^{k\tau}_t  \bar{u}_{p}(t)
{\,{\rm d}x}{\rm d}t
\nonumber
\\
&
=
\frac{1}{k \tau^2}
\int_{(j-1)\tau}^{j \tau}
\int_{\Omega}
\left({b}(\bar{u}_{p}(t+k\tau)) - {b}(\bar{u}_{p}(t))  \right)
\left(\bar{u}_{p}(t+k\tau) - \bar{u}_{p}(t)\right)
{\,{\rm d}x}{\rm d}t.
\end{align*}
Hence, summing over $j=1, \dots, p-k$ we get the estimate
\begin{align}\label{est:401}
\sum_{j=1}^{p-k}
&\int_{j \tau}^{(j+k)\tau}
\int_{\Omega}
\partial_t^{-\tau} {b}(\bar{u}_{p}(t)) \, \partial^{k\tau}_t  \bar{u}_{p}(t)
{\,{\rm d}x}{\rm d}t
\\
&
\geq
\frac{1}{k \tau^2}
\int_{0}^{T - k\tau}
\int_{\Omega}
\left({b}(\bar{u}_{p}(t+k\tau)) - {b}(\bar{u}_{p}(t))  \right)
\left(\bar{u}_{p}(t+k\tau) - \bar{u}_{p}(t)\right)
{\,{\rm d}x}{\rm d}t.
\nonumber
\end{align}
Similarly, for the elliptic term, after a little lengthy but straightforward computation we obtain
\begin{align}\label{est:402}
&\sum_{j=1}^{p-k}
\int_{j \tau}^{(j+k)\tau}
\int_{\Omega}
a(\bar{\theta}_{p}(t-\tau)) \nabla \bar{u}_{p}
\cdot
\nabla \partial^{k\tau}_t \bar{u}_{p}
{\,{\rm d}x}{\rm d}t
\\
&
=
\sum_{\ell=1}^{k} \sum_{j=1}^{p-k}
\int_{(j+\ell-1)\tau}^{(j+\ell)\tau}
\int_{\Omega}
\left( a(\bar{\theta}_{p}(t-\tau))
\nabla
\bar{u}_{p}
\right)
\cdot\nabla
\partial^{k\tau}_t  \bar{u}_{p}
{\rm d}x{\rm d}t
\nonumber
\\
&
=
\sum_{\ell=1}^{k}
\int_{\ell\tau}^{T-k\tau+\ell\tau}
\int_{\Omega}
a(\bar{\theta}_{p}(t-\tau))\nabla \bar{u}_{p}(t)
\cdot\nabla
\partial^{k\tau}_t \bar{u}_{p}(t-\ell\tau)
{\,{\rm d}x}{\rm d}t
\nonumber
\\
&
\leq
\frac{c_1}{\tau}
\int_{Q_T}
|a(\bar{\theta}_{p}(t-\tau))\nabla \bar{u}_{p}|^2
{\,{\rm d}x}{\rm d}t
+
\frac{c_2}{\tau}
\int_{Q_T}
|\nabla \bar{u}_{p}|^2
{\,{\rm d}x}{\rm d}t
\nonumber
\\
&
\leq
\frac{C}{\tau}.
\nonumber
\end{align}

Combining \eqref{est:401}--\eqref{est:402} and using \eqref{energy_estimate_u} we obtain
\begin{align}\label{est:402b}
\int_0^{T-k\tau}
\left({b}(\bar{u}_{p}(s+k\tau)) - {b}(\bar{u}_{p}(s)) \right)
(\bar{u}_{p}(s+k\tau) - \bar{u}_{p}(s))
{\rm d}s
\leq C k \tau.
\end{align}
Using the
compactness argument one can show in the same way
as in \cite[Lemma~1.9]{AltLuckhaus1983} and \cite[Eqs. (2.10)--(2.12)]{FiloKacur1995}
\begin{equation}\label{eq555}
b(\bar{u}_{p})   \rightarrow   b(u) \textmd{ in }L^1(Q_T)
\end{equation}
and almost everywhere on $Q_T$.
Since $b$ is strictly monotone,
it follows from \eqref{eq555}
that \cite[Proposition 3.35]{Kacur1990a}
\begin{equation}
\label{conv:u00}
\bar{u}_{p}  \rightarrow  u
\qquad
\textrm{ almost everywhere on } Q_T.
\end{equation}

Further,
in much the same way as
in \eqref{est:402b},      we arrive at
\begin{equation}\label{est21_w}
\int_0^{T-k\tau}
|{b}(\bar{u}_{p}(s+k\tau))\bar{w}_{p}(s+k\tau) - {b}(\bar{u}_{p}(s))\bar{w}_{p}(s) |^2
{\rm d}s
\leq C k \tau.
\end{equation}
From this we conclude, using \eqref{est:unform_bound_w_05}, that
\begin{equation}\label{est22_w}
\int_0^{T-k\tau}
|\bar{w}_{p}(s+k\tau) - \bar{w}_{p}(s) |^2
{\rm d}s
\leq C k \tau.
\end{equation}

Finally, in a similar way, using \eqref{est:unform_bound_theta_05}, we arrive at
\begin{equation}\label{est21_theta}
\int_0^{T-k\tau}
|\bar{\theta}_{p}(s+k\tau) - \bar{\theta}_{p}(s)|^2
{\rm d}s
\leq {C} k \tau.
\end{equation}

\subsection{Passage to the limit}
\label{subsec:limit}

The a-priori estimates
\eqref{est:unform_bound_w_05},
\eqref{est:unform_bound_theta_05},
\eqref{energy_estimate_u},
\eqref{est20_w_a},
\eqref{est20_theta_a},
\eqref{est:402b},
\eqref{est22_w},
\eqref{est21_theta}
allow us to conclude that there exist
$u \in L^2(I;W^{1,2}_{\Gamma_D}(\Omega))$,
$w \in L^2(I;W^{1,2}_{\Gamma_D}(\Omega)) \cap L^{\infty}({Q_T})$
and
$\theta \in L^2(I;W^{1,2}_{\Gamma_D}(\Omega)) \cap L^{\infty}({Q_T})$ such that,
letting $p \rightarrow +\infty$ (along a selected subsequence),
\begin{align*}
\bar{u}_{p} & \rightharpoonup  u
&&
\textrm{weakly in } L^2(I;W^{1,2}_{\Gamma_D}(\Omega)),
\\
\bar{u}_{p} & \rightarrow u
&&
\textrm{almost everywhere on } Q_T,
\\
\bar{w}_{p} & \rightharpoonup  w
&&
\textrm{weakly in } L^2(I;W^{1,2}_{\Gamma_D}(\Omega)),
\\
\bar{w}_{p} & \rightharpoonup  w
&&
\textrm{weakly star in } L^{\infty}({Q_T}),
\\
\bar{w}_{p} & \rightarrow w
&&
\textrm{almost everywhere on } Q_T,
\\
\bar{\theta}_{p} & \rightharpoonup  \theta
&&
\textrm{weakly in } L^2(I;W^{1,2}_{\Gamma_D}(\Omega)),
\\
\bar{\theta}_{p}  & \rightharpoonup  \theta
&&
\textrm{weakly star in } L^{\infty}({Q_T}),
\\
\bar{\theta}_{p} & \rightarrow \theta
&&
\textrm{almost everywhere on } Q_T.
\end{align*}
The above established convergences
are sufficient for taking the limit $p \rightarrow \infty$
in \eqref{eq:101}--\eqref{eq:103}
(along a selected subsequence) to get the weak solution
of the system \eqref{strong:eq1a}--\eqref{strong:eq1l}
in the sense of Definition~\ref{def_weak_solution}.
This completes the proof of the main result stated by Theorem~\ref{main_result}.

\bigskip

\paragraph{Acknowledgment}
This research was supported by the project GA\v{C}R~16-20008S
(Michal Bene\v{s}) and by the grant SGS16/001/OHK1/1T/11 provided
by the Grant Agency of the Czech Technical University in Prague (Luk\'{a}\v{s} Krupi\v{c}ka).

\bibliographystyle{amsplain}

\end{document}